\documentclass[
final
]{dmtcs-episciences}

\usepackage[utf8]{inputenc}
\usepackage{subfigure}

\usepackage[round]{natbib}
\usepackage{color}
\usepackage{amsmath}

\newtheorem{theorem}{Theorem}

\newtheorem{lemma}[theorem]{Lemma}

\newtheorem{proposition}[theorem]{Proposition}

\newtheorem{remark}[theorem]{Remark}
\newtheorem{claim}[theorem]{Claim}

\newcounter{mycount}

\def\0{\mathbf{0}}

\author{G\"{u}lnaz Boruzanl{\i} Ekinci\affiliationmark{1}
  \and John Baptist Gauci\affiliationmark{2}
  }
\title[The super-connectivity of Johnson graphs]{The super-connectivity of Johnson graphs}
\affiliation{
  Ege University, Turkey\\
  University of Malta, Malta
  }
\received{2019-6-26}
\revised{2020-2-25}
\accepted{2020-2-26}
\begin{document}
\publicationdetails{22}{2020}{1}{12}{5600}
\maketitle
\begin{abstract}
     For positive integers $n,k$  and $t$, the uniform subset graph $G(n, k, t)$ has all $k$--subsets of $\{1,2,\ldots, n\}$ as vertices and two $k$--subsets are joined by an edge if they intersect at exactly $t$ elements. The Johnson graph $J(n,k)$ corresponds to $G(n,k,k-1)$, that is, two vertices of $J(n,k)$ are adjacent if the intersection of the corresponding $k$-subsets has size \break $k-1$. A super vertex--cut of a connected graph is a set of vertices whose removal disconnects the graph without isolating a vertex and the super--connectivity is the size of a minimum super vertex--cut. In this work, we fully determine the  super--connectivity of the family of Johnson graphs $J(n,k)$ for $n\geq k\geq 1$.
\end{abstract}

\section{Introduction} \label{Intro}

Let $n$ and $k$ be integers such that $n\geq k\geq 1$ and let $[n]$ denote the set $\{1,2,\ldots,n\}$. The \emph{Johnson graph} $J(n,k)$ is the graph with vertex set $V(J(n,k))$ consisting of all the $k$--subsets of $[n]$ and with edge set $E(J(n,k))=\{\{u,v\} : u,v\in V(J(n,k)) {\rm ~and~} |u\cap v|=k-1\}$. It is well-known that $J(n,k) $ is regular of degree $ k(n-k) $ and that $J(n,1)$ is isomorphic to the complete graph on $n$ vertices.

Johnson graphs appear in the theory of association schemes (see \cite{Moon1984}) and are a particular instance of the more general uniform subset graphs $G(n,k,t)$ introduced by \cite{ChenLih1987} corresponding to the case when $t=k-1$. Special cases of the uniform subset graphs have been investigated extensively for a variety of parameters, such as girth, diameter, Hamiltonicity and connectivity (see for example \cite{Agong2018, GauciBoruzanli, ChenWeng2008, Mutze2017, Simpson1994}). The quasipolynomial algorithm for graph isomorphism by \cite{Babai2016} has recently put Johnson graphs in the limelight, especially within the computer science community. This family of graphs present, in fact, the only obstructions to effective partitioning. Johnson graphs were also studied by  \cite{Diego2018problem} for their isoperimetric function.

Let $ G $ be a graph with vertex set $ V(G) $ and edge set $ E(G) $. \cite{Harary1983} proposed the notion of \emph{conditional connectivity} which asks for the size of a minimum vertex--cut $S$ of $G$, if it exists, so that $G-S$ is disconnected and every component of the resulting graph $G-S$ has some graph theoretical property $P$. We recall that a \emph{vertex--cut} $S$ of a graph $G$ is a set of vertices of $G$ whose deletion results in a disconnected graph or leaves an isolated vertex. A \emph{minimum vertex--cut} is one of smallest cardinality over all vertex--cuts of $G$ and the \emph{connectivity} $\kappa=\kappa(G)$ of $G$ is the size of a minimum vertex--cut.

Motivated by Harary's notion, many researchers studied various types of conditional connectivity. The case when the condition is that every resulting component is not an isolated vertex gave rise to what became to be known as the super--connectivity of a graph. More precisely, the \emph{super--connectivity} $\kappa'=\kappa'(G)$ of a graph $G$ is the size of a minimum vertex--cut $S$ such that $G-S$ has no isolated vertices. If such a vertex--cut exists, it is referred to as a \emph{super vertex--cut}; otherwise we write $ \kappa'(G) = +\infty $. The super--connectivity $\kappa'$ is of particular interest in cases when $G$ is super--connected, since otherwise $\kappa'=\kappa$. A graph $ G $ is \emph{super--connected} if every minimum vertex--cut is composed of the neighbourhood $ N_G(x) $ of a vertex $ x \in V(G) $, where  $ N_G(x) = \{ y \in V(G): xy \in E(G)\} $. Some examples of graph classes which have been analysed for their super--connectivity are circulant graphs (\cite{BoeschTindell1984}), hypercubes (\cite{Guo2018, YangMeng2009,YangMeng2010}), products of various graphs (see \cite{Boruzanli2016, Cao2014, Guo2015, LuEtAl2008}, and the references therein), generalized Petersen graphs (\cite{BoruzanliGauci-GP}), minimal Cayley graphs (\cite{Hamidoune1999iso}) and Kneser graphs (\cite{Balbuena2019, GauciBoruzanli}). In this work we analyse and establish the super--connectivity of Johnson graphs. Trivially for $ n=k$, the Johnson graph $ J(n,k) $ is composed of an isolated vertex; hence we consider $ n\geq k+1 $.

The applicability of Johnson graphs to the design of networks has further contributed towards their popularity within the scientific community, especially because they have a small diameter and high connectivity (see, for example, \cite{Bautista2013}). Johnson graphs are also both vertex--transitive and edge--transitive. Whereas vertex--transitivity permits the implementation of the same routing and communication schemes at each vertex (or node) of the network, edge--transitivity allows recursive constructions to be used. This is the main reason why symmetric graphs are usually preferred when modelling interconnection networks (\cite{Heydemann1997}). Another characteristic that is generally sought for in networks is regularity because this property simplifies the study of networks in terms of diameter and diameter vulnerability problems. Thus, our choice to study Johnson graphs has also a functional aspect as these graphs can be applied to network designs.

\cite{Watkins1970} showed that if $ G $ is an edge--transitive graph and if the minimum degree of $G$ is $\delta$, then $ \kappa(G)=\delta $.  \cite{Meng2003} proved that a connected vertex--transitive and edge--transitive graph is not super--connected if and only if it is isomorphic to the lexicographic product of a cycle $ C_n (n\geq6) $ or the line graph $ L(Q_3) $ of the cube $ Q_3 $ by a null graph $ N_m $. Furthermore, \cite{Brouwer2009vertex} proved that for a non--complete distance--regular graph of degree $ k $, the connectivity   {is equal to} $ k $, and the only disconnecting sets of size $ k $ are the sets of neighbours of a vertex. Note that Johnson graphs are distance--transitive and hence distance--regular.  The previous two arguments imply that Johnson graphs are super--connected. Consequently, this class of graphs is more important for application purposes as it has been argued by many that a network is more reliable if it is super--connected (see, for example, \cite{HsuLin2008}). As $ \kappa'(G) > \kappa(G) $ for super--connected graphs,  it is natural to ask what the super--connectivity of the Johnson graph $ J(n,k) $ is. Moreover, \cite{Van2014Dist} asked about the minimum number of vertices that need to be deleted to disconnect a distance-regular graph with diameter at least three such that each resulting component has at least two vertices (Problem 41). In Theorem 10, we verify their claimed value in the case of Johnson graphs.

In the next section, we consider the class of Johnson graphs with the smallest value of $k$ for which the super--connectivity exists, namely $J(n,k)$ when $k=2$. In Section \ref{Sect:last}, we determine the super--connectivity of $J(n,k)$ when $k\geq 3$. The line of thought used in the former case cannot be generalised to prove the latter case; hence the reason why we present the two results separately. To avoid confusion, a vertex $x$ of $J(n,k)$ corresponding to the $k$--subset $\{1,2,\ldots,k\}$  will be denoted by $x=z_1z_2\ldots z_k$, where $ z_1, \dots, z_k $ are referred to as the entries of $ x $. The \emph{Hamming distance}, denoted by $ H(x,y) $, is the number of entries that differ between two vertices $ x $ and $ y $ of $ J(n,k) $. Clearly, if $ x $ is adjacent to $ y $ in $ J(n,k) $, then $ H(x,y) = 1 $. For notation and terminology not defined here, we refer the reader to \cite{Chartrand2010}.
\section{The super--connectivity of $J(n,2)$} \label{Sect-mainResults}

It is easy to see that $J(3,2)$, $J(4,2)$ and $J(5,2)$ do not have a super vertex--cut. Therefore, in this section we discuss the super--connectivity of $J(n,2)$, where $ n\geq 6 $.

The result of the following theorem was also obtained by \cite{Cioabua2012conj}. The reason why we are including the proof below is two-fold. First, the proof by \cite{Cioabua2012conj} uses a different approach than the one adopted below, in that they conduct a case analysis based on cliques. Secondly, the proof below should assist the reader in getting accustomed to the notation being used in the subsequent proofs.
\begin{theorem}
	$ \kappa'(J(n,2)) =   3(n-3) $, where $ n\geq 6 $.
\end{theorem}

\begin{proof}
	Let $ \mathcal{G}=J(n,2) $, where $ n\geq 6 $. Let $ S$ be a super vertex--cut of $ \mathcal{G} $ and suppose, for contradiction, that $ |S|<3(n-3) $.
	
	Since $S$ is a super vertex--cut, each component of $\mathcal{G}-S$ contains at least two adjacent vertices sharing a common entry. Without loss of generality, assume that a component, say $ C_1 $, contains the two adjacent vertices $ z_1z_2 $ and $ z_1z_3 $. The two adjacent vertices of another component, say $ C_2 $, cannot share any common entries with the vertices of $ C_1 $, so, without loss of generality, let $ z_{n-2}z_{n-1} $ and $ z_{n-2}z_{n} $ be two adjacent vertices in $ C_2 $. Note that every vertex--cut contains the common neighbours of the vertices in different components. Letting $ S' = \{z_{\alpha}z_{\beta}: \alpha\in \{1,2,3\} \textrm{ and } \beta\in \{n-2,n-1,n\}\}$, then $ S $ contains $ S' $. If $n=6$, then $|S'|=9>|S|$ and we get the required contradiction. Thus, in the sequel we can assume that $n\geq 7$.
	
	We now consider the vertices in the following two sets:
	\begin{enumerate}[$\bullet$]\setlength{\itemsep}{-5pt}
		\item[] $A=\Big\{ z_{\alpha}z_{\gamma} : \alpha \in \{1,2,3\} \textrm{ and } \gamma \notin \{1,2,3,{n-2},{n-1},n\} \Big\}$, and \vspace{0.2cm}
		\item[] $B=\Big\{ z_{\delta}z_{\beta}  : \delta \notin \{1,2,3,{n-2},{n-1},n\}\textrm{ and } \beta \in \{{n-2},{n-1},n\}\Big\} $.
	\end{enumerate}
	
	For every $ i\in \{4,5,\dots,n-3\} $, let $ A_i \subseteq A$ and $ B_i \subseteq B $ be such that  $ A_i = \left\{ z_{\alpha}z_i : \alpha \in \{1,2,3\} \right\}$ and $ B_i = \left\{ z_iz_{\beta} : \beta \in \{{n-2},{n-1},n\} \right\}$. Thus $ |A_i|=|B_i|=3 $, for all $ i \in \{4,5,\dots,n-3\} $ and $ |A|=|B|=3(n-6) $. We note that each vertex of $ A $ not in $ S $ is adjacent to some vertex in $ C_1 $ and, similarly, each vertex of $ B $ not in $ S $ is adjacent to some vertex in $ C_2 $. Also, each vertex of $ A_i $ is adjacent to each vertex of $ B_i $.
	
	If at least $ 3(n-6) $ vertices of $ A \cup B $ are in $ S $, then
	$|S| \geq |S'| + 3(n-6)=3(n-3)$, a contradiction. Thus, at most $3(n-6)-1$ vertices of $A\cup B$ are in $S$, implying that at least $ 3(n-6)+1 $ vertices of $ A\cup B $ are in $ \mathcal{G}-S $. Thus, there is at least one $ i\in \{4,5,\dots,n-3\} $ for which at least a vertex of both $ A_i $ and $ B_i $ is in $\mathcal{G}-S$, and hence there is an edge connecting $ C_1 $ and $ C_2$, a contradiction. Therefore, $ |S| \geq 3(n-3) $.

	Finally, if we consider any three distinct entries $ \{z_i,z_j,z_k\}$ for $ i,j,k \in [n] $, the set $ \{z_{\alpha}z_{\beta}: \alpha\in \{i,j,k\} \textrm{ and } \beta \in [n]\setminus \{i,j,k\} \}$ is a vertex--cut of $ \mathcal{G} $ which does not create isolated vertices, and hence forms a super vertex--cut of $ \mathcal{G} $ of cardinality $ 3(n-3) $.
	
\end{proof}

\section{The super--connectivity of $J(n,k)$ for $k\geq 3$} \label{Sect:last}

\cite{Daven1999} gave a proof that does not rely on the edge--transitivity property of Johnson graphs to show that the connectivity of $J(n,k)$ is equal to the degree. This proof involves induction and is based on the observation that the graph $ J(n,k) - \beta_r $ is isomorphic to $ J(n-1,k) $ for any $ r \in [n] $, where $ \beta_r $ is the set of all vertices containing the entry $ z_r $. We adopt the general approach used by \cite{Daven1999} to establish $\kappa'(J(n,k))$, although in our case the work required is  more involved due to the nature of the parameter being studied. To simplify the notation used, in the sequel  we shall denote by $x_i^j$ the vertex obtained from the vertex $x$ by removing the entry $z_i$ from $x$ and introducing a new entry $z_j$ which is not in $x$. For instance, if $x=z_1z_2\ldots z_k$, then $x_1^{k+1}=z_2\ldots z_kz_{k+1}$. This process can be repeated iteratively in such a way that $x_{i,h}^{j,\ell}=\left(x_i^j\right)_h^{\ell}$, where the entry $ z_h $ is in $x_i^j$ and the entry $ z_\ell $ is not in $x_i^j$.

It is useful to note that a minimum super vertex--cut $S$ of a connected graph $G$ contains a vertex $v$ which has at least one neighbour in some component of $G-S$, since otherwise $G$ is disconnected. Suppose now that there is a component of $G-S$ that does not contain any neighbour of this vertex $v$. Then the set of vertices $T=S\setminus\{v\}$ is also a super vertex--cut of $G$ because $G-T$ is disconnected and contains no isolated vertices. However, this contradicts the minimality of $S$, and hence the following remark follows immediately.

\begin{remark} \label{Rem-NborEveryComp}
	A minimum super vertex--cut $S$ of $G$ contains a vertex having at least one neighbour in every component of $G-S$. Moreover, if a vertex $v$ in a minimum super vertex--cut $S$ of $G$ has a neighbour in one component of $G-S$, then it has at least one neighbour in every component of $G-S$.
\end{remark}

We require Lemmas \ref{lemma3}, \ref{lemma4} and \ref{lemma5} in the rest of this work, the proofs of which can be adapted from \cite{Daven1999} and hence we omit them. Although these results were originally proved for the connectivity of $J(n,k)$, we observe that the proofs are very similar in the case of the super--connectivity. Whereas in the original connectivity version the argument used in the proofs revolves around the fact that every vertex in the vertex--cut $S$ has a neighbour in each of the components of $G-S$, in the case of super--connectivity the main tool required is given by Remark \ref{Rem-NborEveryComp} above, but the rest of the proof then follows. For the sake of consistency, we have also changed the notation in the statements to be in line with ours.

\begin{lemma}  \label{lemma3}
	If $ S $ is a minimum super vertex--cut of $ J(n,k) $, where  $k\geq 3$ and
	$n\geq k+3$, then    the entry $ z_r $ is contained in at least one vertex of $ S $ for every $ r\in [n] $.
\end{lemma}

\begin{lemma}  \label{lemma4}
	If $ S $ is a minimum super vertex--cut of $  J(n,k) $, where  $k\geq 3$ and
	$n\geq k+3$, then no entry $ z_r $ is contained in every vertex of $ S $, where $ r\in [n]$.
\end{lemma}

\begin{lemma}  \label{lemma5}
	Let $ S $ be a minimum super vertex--cut of $ \mathcal{G} = J(n,k) $, where  $k\geq 3$ and
	$n\geq k+3$. Denote a smallest component of $ \mathcal{G}-S $ by $ C $ and let $ C_*=(\mathcal{G}-S)-C $. There exists an entry $ z_r $, for $ r\in [n] $, such that $ C - \beta_r \neq \emptyset $ and $ C_* - \beta_r \neq \emptyset $.
\end{lemma}

The upper bound for the super--connectivity of $J(n,k)$ is given in Lemma \ref{Lem:UB-J(n,k)}, which will then be used in Theorems \ref{Theo:k=3} and \ref{Theo:main} to establish the equality.

\begin{lemma} \label{Lem:UB-J(n,k)}
	$\kappa'(J(n,k))\leq (2k-1)(n-k)-k$ for $k\geq 3$ and
	$n\geq k+3$.
\end{lemma}

\begin{proof}
	Consider any two adjacent vertices in $J(n,k)$, say $x=z_1z_2\ldots z_k$ and $x_k^{k+1}$, and let $S$ be the set of all their neighbours. Then $S$ is the union of the following mutually disjoint sets of vertices:
	\begin{itemize}
		\item the set $S_1$ composed of the common neighbours of $x$ and $x_k^{k+1}$, that is
		$$S_1=\Big\{x_i^{k+1} : i\in\{1,\ldots,k-1\}\Big\} \cup \Big\{x_k^j : j\in\{k+2,\ldots,n\}\Big\};$$
		\item the set $S_2$ composed of the neighbours of $x$ which are not also neighbours of $x_k^{k+1}$, that is
		$$S_2=\left\{x_i^j : i\in\{1,\ldots,k-1\} {\rm ~~and~~} j\in\{k+2,\ldots,n\}\right\};$$
		\item the set $S_3$ composed of the neighbours of $x_k^{k+1}$ which are not also neighbours of $x$, that is
		$$S_3=\left\{x_{i,k}^{k+1,j} : i\in\{1,\ldots,k-1\} {\rm ~~and~~} j\in\{k+2,\ldots,n\}\right\}.$$
	\end{itemize}
	
	Thus,
	\begin{align*}
	|S| &= |S_1|+|S_2|+|S_3|\\
	&= \left((k-1)+(n-k-1)\right)+2(k-1)(n-k-1)=(2k-1)(n-k)-k.
	\end{align*}
	
	It is easy to see that there is no vertex in $ J(n,k)-S $ that is adjacent to $ k(n-k) $ vertices in $ S $, and hence $ J(n,k)-S $ has no isolated vertices. Thus, $ S $ is a super vertex-cut and the bound follows.
\end{proof}

The following two lemmas provide the essential tools required in proving the two main theorems in this section, namely Theorems \ref{Theo:k=3} and \ref{Theo:main}.

\begin{lemma} \label{Lem:AllNborsContainR}
	Let $S$ be a minimum super vertex--cut of $\mathcal{G}=J(n,k)$, where $k\geq 3$ and $n\geq k+3$. If there is a vertex $x\in V(\mathcal{G}-S)-\beta_r$, for some $r\in [n]$, such that $N_{\mathcal{G}-S}(x)-\beta_r = \emptyset$, then $|S|\geq (2k-1)(n-k)-k$.
\end{lemma}

\begin{proof}
	For $k\geq 3$ and $n\geq k+3$ , let $S$ be a minimum super vertex--cut of $\mathcal{G}=J(n,k)$. We fix a value of $r$, say $r=n$ and let $x=z_1z_2\ldots z_k\in V(\mathcal{G}-S)-\beta_n$. We note that if $N_{\mathcal{G}-S}(x)-\beta_n=\emptyset$, then all the neighbours of $x$ in the graph $\mathcal{G}-S$ contain the entry $z_n$, that is,
	\begin{equation} \label{Eqn:Nx}
	N_\mathcal{G}(x)-\beta_n\subseteq S.
	\end{equation}
	Also, since $|N_\mathcal{G}(x)\cap\beta_n|=k$, then $1\leq |N_{\mathcal{G}-S}(x)\cap \beta_n|\leq k$. We may thus assume that, without loss of generality, $x_k^n\in V(\mathcal{G}-S)$, and hence $x_k^n$ is in the same component of $\mathcal{G}-S$ as $x$, say $C_1$.
	
	We let $S_1$ be the set of all the neighbours of $x$ that do not contain the entry $z_n$, that is, $S_1=N_G(x)-\beta_n$. Therefore,
	\begin{equation} \label{Eqn:S1}
	|S_1|=|N_\mathcal{G}(x)|-|N_\mathcal{G}(x)\cap \beta_n|=k(n-k)-k=k(n-k-1).
	\end{equation}
	We remark that $S_1=\left\{x_i^j : i\in\{1,\ldots,k\} {\textrm{ and }} j\in\{k+1,\ldots,n-1\}\right\}$, and that by (\ref{Eqn:Nx}), we have $S_1\subseteq S$.
	From Remark \ref{Rem-NborEveryComp}, it follows that all the vertices in $S_1$ have at least one neighbour in each component of $\mathcal{G}-S$. In particular, consider the vertex $x_k^{k+1}\in S_1 \subseteq S$ and take one of its neighbours $w$ in $\mathcal{G}-S$ which is not in $C_1$. The vertex $w$ contains the entry $z_{k+1}$, for otherwise it would be in $S_1$. Also, $w$ contains an entry $z_h$ for $h\in\{k+2,\ldots,n-1\}$, for otherwise, if $w=x_i^{k+1}$ for $i\in\{1,\ldots,k-1\}$ then $w\in S_1$, and if $w=x_{i,k}^{k+1,n}$ for $i\in\{1,\ldots,k-1\}$, then $w$ is adjacent to $x_k^n$ in $C_1$, a contradiction. Without loss of generality, we can thus assume that $w=x_{k-1,k}^{k+1,k+2}$ and that $w$ is in the component $C_2$ of $\mathcal{G}-S$.
	
	We now consider the neighbours of $w$ which are not in $S_1$. These are given by
	\begin{align*}
	N_{\mathcal{G}-S_1}(w)&=\left\{x_{k-1,k}^{k+1,j} : j\in\{k+3,\ldots,n\}\right\} \cup \left\{x_{k-1,k}^{k+2,j} : j\in\{k+3,\ldots,n\}\right\} \\
	& \cup \left\{x_{i,k-1,k}^{k+1,k+2,j} : i\in\{1,\ldots,k-2\} {\rm ~and~ } j\in\{k-1,k,k+3,\ldots,n\}\right\}.
	\end{align*}
	
	We partition this set of $k(n-k)-4$ vertices into four sets, as follows:
	
	$A_1=\Big\{v\in N_{\mathcal{G}-S_1}(w):H(v,x_k^n)=1\Big\}=\left\{x_{k-1,k}^{k+1,n},x_{k-1,k}^{k+2,n}\right\}$, implying that $|A_1|=2$.
	
	\begin{align*}
	A_2 &= \Big\{v\in N_{\mathcal{G}-S_1}(w):H(v,x_k^n)=2\Big\}\setminus\left\{x_{k-1,k}^{k+1,j}:j\in\{k+3,\ldots,n-1\}\right\} \\
	&= \left\{x_{k-1,k}^{k+2,j}:j\in\{k+3,\ldots,n-1\}\right\}\cup\left\{x_{i,k}^{k+1,k+2}:i\in\{1,\ldots,k-2\}\right\} \\
	& ~~~~\cup \left\{x_{i,k-1,k}^{k+1,k+2,n}:i\in\{1,\ldots,k-2\}\right\},
	\end{align*}
	implying that  $|A_2|=n+k-7$.
	
	\begin{align*}
	A_3 &= \Big\{v\in N_{\mathcal{G}-S_1}(w):H(v,x_k^n)=3\Big\} \\
	&= \left\{x_{i,k-1,k}^{k+1,k+2,j}:i\in\{1,\ldots,k-2\} {\rm ~and~ } j\in\{k,k+3,\ldots,n-1\}\right\},
	\end{align*}
	implying that $|A_3|=(k-2)(n-k-2)$.
	
	$A_4=\left\{x_{k-1,k}^{k+1,j}:j\in\{k+3,\ldots,n-1\}\right\}$, implying that $|A_4|=n-k-3$.

	We construct internally disjoint paths from $w$ to $x_k^n$ or from $w$ to $x$, depending on which one of the sets defined above contains the different neighbours of $w$.
	
	\begin{enumerate}[(I)]
		\item   For the neighbours of $w$ which are in $A_1$, the paths are as follows:
		$$w\sim x_{k-1,k}^{k+1,n} \sim x_k^n ~~{\rm and }~~ w\sim x_{k-1,k}^{k+2,n} \sim x_k^n.$$
		\item   For the neighbours of $w$ which are in $A_2$, the paths are as follows:
		\begin{enumerate}[(i)]
			\item for $j\in\{k+3,\ldots,n-1\}$,
			$$w\sim x_{k-1,k}^{k+2,j} \sim x_{k-1,k}^{n,j} \sim x_k^n;$$
			\item for $i\in\{1,\ldots,k-2\}$,
			$$w \sim x_{i,k}^{k+1,k+2} \sim x_{i,k}^{k+1,n} \sim x_k^n;$$
			\item for $i\in\{1,\ldots,k-2\}$,
			$$w \sim x_{i,k-1,k}^{k+1,k+2,n} \sim x_{i,k}^{k+2,n} \sim x_k^n.$$
		\end{enumerate}
		\item   For the neighbours of $w$ which are in $A_3$, then for $i\in\{1,\ldots,k-2\}$ and $j\in\{k,k+3,\ldots,n-1\}$, the paths are given by
		$$w\sim x_{i,k-1,k}^{k+1,k+2,j} \sim x_{i,k-1,k}^{k+1,n,j} \sim x_{i,k}^{n,j} \sim x_k^n.$$

		\item   Finally, we consider one of the vertices in $A_4$, say $x_{k-1,k}^{k+1,k+3}$, and construct a path between $x$ and $w$ as follows
		$$w\sim x_{k-1,k}^{k+1,k+3} \sim x_{2,k-1}^{k+1,k+3} \sim x_{2,k-1}^{k+3,n} \sim x_{k-1}^n \sim x.$$
		
	\end{enumerate}
	It is easy to check that the vertices utilised in constructing all the above paths are not in $S_1$ and are all distinct. Thus, the number of vertex disjoint paths constructed from the vertices $x$ and $x_k^n$ in $C_1$ to the vertex $w$ in $C_2$ is given by
	$$|A_1|+|A_2|+|A_3|+1=2+(n+k-7)+(k-2)(n-k-2)+1=(k-1)(n-k).$$
	
	Hence, $S$ contains an additional $(k-1)(n-k)$ vertices and from (\ref{Eqn:S1}), we get that $|S|\geq (n-k)(2k-1)-k$, as required.
\end{proof}

We remark that the assumptions made in the following lemma will seem artificial for the time being, however their relevance will become clear when Lemma \ref{Lem:S-S'} is used in the proof of Theorem \ref{Theo:main}.

\begin{lemma} \label{Lem:S-S'}
	Let $S$ be a minimum super vertex-cut of $\mathcal{G}=J(n,k)$, where $ k\geq 3 $ and $ n \geq k+3 $. Let $ C_1 $ and $ C_2 $ be two components of $ \mathcal{G}-S $ such that $ C_1' = C_1 - \beta_r $ and $ C_2' = C_2- \beta_r$, for some $r\in [n]$. Assume that the vertex $u$ is in $C_1'$ and that the vertex $v$ is in $C_2'$ such that $H(u,v)=2$, and assume that $u$ is adjacent to $\bar{u}$ in $C_1'$ and that $v$ is adjacent to $\bar{v}$ is $C_2'$. Then there are at least $(2k-1)$ internally disjoint paths between $ C_1' $ and $ C_2' $ such that all the internal vertices on these paths contain the entry $z_r$.
\end{lemma}

\begin{proof}
	Since by our assumption $H(u,v)=2$, we consider a common neighbour of $u$ and $v$, which for ease of notation, we denote by $x=z_1z_2\ldots z_k$. This implies that $r\in [n]\setminus [k]$, for otherwise $u$ and $v$ are adjacent. Without loss of generality, let $r=n$. Thus, $u=x_{h_1}^{\ell_1}$ where $h_1\in [k]$ and $\ell_1\in [n-1]\setminus [k]$ and $v=x_{h_2}^{\ell_2}$ where $h_2\in [k]$ and $\ell_2\in [n-1]\setminus [k]$, such that $h_1\neq h_2$ and $\ell_1\neq \ell_2$. For simplicity, in the following  work, whenever we require to remove an entry of $x$ or add an entry which is not in $x$, we shall choose the smallest entry available. Thus, we let $u=x_1^{k+1}$ and $v=x_2^{k+2}$. Consequently, when the entry $z_{k+3}$ is added at some stage in sequel, it is being implicitly assumed that this was chosen without loss of generality to represent $z_{\zeta}$ where $\zeta\in [n-1]\setminus [k+2]$.
	
	For $j\in [k]\setminus\{1,2\}$, consider the vertices
	\begin{align*}
	x_{1,j}^{k+1,n}, x_{2,j}^{k+1,n}, \textrm{~and~} x_{2,j}^{k+2,n}.
	\end{align*}

	For $k\geq 5$, let the paths $P_j$, where $j\in \{5,\ldots,k\}$, be given by
	\begin{align} \label{EqnStar}
	P_j:= u \sim x_{1,j}^{k+1,n} \sim x_{2,j}^{k+1,n} \sim x_{2,j}^{k+2,n} \sim v.
	\end{align}
	
	We remark that for $k\in \{3,4\}$, no paths are constructed thus far. Also, in the cases discussed below, we shall refer again to (\ref{EqnStar}) to construct the path $P_3$ by taking $j=3$ when $k=3$, and the paths $P_3$ and $P_4$ by taking $j\in\{3,4\}$ when $k\geq 4$. For ease of writing, it will thus be assumed that when we say that we put $j=3$ and $j=4$ in (\ref{EqnStar}) to construct paths $P_3$ and $P_4$, these paths are actually constructed for each of the cases $k=3$ and $k\geq 4$ by substituting the corresponding values of $j$ in (\ref{EqnStar}) according to the explanation given above.
	
	Furthermore, define the paths $T_1$ and $T_2$ as follows:
	\begin{align*}
	T_1 &:= u \sim x_1^n \sim x_{1,2}^{k+2,n} \sim v \\
	T_2 &:= u \sim x_{1,2}^{k+1,n} \sim x_2^n \sim v.
	\end{align*}
	
	Then the paths $P_j$ for $ j \in \{5,\dots,k\} $  (if any), together with the two paths $T_1$ and $T_2$ are internally disjoint paths. We note that the pairs of entries $\{z_1,z_2\}$, $\{z_{k+1},z_{k+2}\}$, and $\{z_2,z_{k+2}\}$ do not feature in any of the internal vertices of the paths constructed thus far. In the sequel, we refer to these pairs of entries as ``the missing pairs of entries''. Also, all the internal vertices of these paths contain the pair of entries $\{z_3,z_4\}$. These two observations make it easier to verify that the internal vertices that are used to construct any further paths are different from those already availed of.

	In the rest of the proof, we will show that we can construct  the remaining required number of internally disjoint paths from $\{u,\bar{u}\}$ to $\{v,\bar{v}\}$ such that all the internal vertices of these paths contain the entry $z_n$.  Since exactly one entry of $u$ is not in $\bar{u}$ and similarly one entry of $v$ is not in $\bar{v}$, and noting that there are no edges from the vertices of $C_1$ to the vertices of $C_2$, then the following properties hold:
	\begin{enumerate}[(i)]
		\item   $1\leq H(\bar{u},x) \leq 2$ and $1\leq H(\bar{v},x) \leq 2$; and
		\item   $2\leq H(\bar{u},\bar{v})\leq 4$.
	\end{enumerate}
	The above properties underlie the motivation behind the four separate cases that follow, in the sense that each depends on the entries of $x$ that are present in the vertices $\bar{u}$ and $\bar{v}$. We thus focus on the entries of $x$ that are in both $u$ and $v$, namely $\{z_3,\ldots,z_k\}$, and differentiate between the cases when these entries are or are not in $\bar{u}$ and $\bar{v}$.
	
	\noindent\emph{Case I.} The entries $\{z_3,z_4,\ldots,z_k\}$ are in both $\bar{u}$ and $\bar{v}$.
	
	In this case, the vertex $\bar{u}$ is given by $x_{1,2}^{\beta,k+3}$ where $\beta\in\{2,k+1\}$, and the vertex $\bar{v}$ is given by $x_{1,2}^{\alpha,k+4}$ where $\alpha\in\{1,k+2\}$. We remark that, as explained above, the entries $z_{k+3}$ and $z_{k+4}$ were chosen without loss of generality to simplify the work done.
	
	For $i\in\{3,\ldots,k\}$, let the paths $Q_i$ be given by
	$$Q_i:= \bar{u} \sim x_{1,2,i}^{\beta,k+3,n} \sim x_{1,2,i}^{\alpha,k+3,n} \sim x_{1,2,i}^{\alpha,k+4,n} \sim \bar{v}$$
	and the path $T_3$ be
	$$T_3:= \bar{u} \sim x_{1,2}^{k+3,n} \sim x_{1,2}^{k+4,n} \sim \bar{v}.$$
	
	Finally, we put $j=3$ and $j=4$ in (\ref{EqnStar}) to obtain the internally disjoint paths $P_3$ and $P_4$, which are also internally disjoint from the paths $P_j$ for $j\in\{5,\ldots,k\}$ (if any).

	We note that the paths $Q_i$ and $T_3$ are $k-1$ internally disjoint paths, and since all their internal vertices contain either the entry $z_{k+3}$ or $z_{k+4}$, they are also internally disjoint from the paths $T_1$, $T_2$ and $P_j$ for $j\in\{3,\ldots,k\}$. We have thus constructed a total of $2k-1$ internally disjoint paths from the vertices in $C_1$ to the vertices in $C_2$, as required.

	\noindent\emph{Case II.} The entries $\{z_3,\ldots,z_k\}$ are in $\bar{v}$ but not in $\bar{u}$ (or vice-versa).
	
	Since $H(\bar{u},u)=1$, then at most one of $\{z_3,\ldots,z_k\}$ is not in $\bar{u}$, say $z_3$. Thus, the vertex $\bar{u}$ is given by $x_{1,3}^{k+1,\beta}$ where $\beta\in\{1,k+2,k+3\}$, and the vertex $\bar{v}$ is given by $x_{1,2}^{\gamma,\alpha}$ where $\gamma\in\{1,k+2\}$ and $\alpha\in\{k+3,k+4\}$.
	
	For $i\in\{4,\ldots,k\}$, let the paths $Q_i$ (if any) be given by
	$$Q_i:= \bar{u} \sim x_{1,3,i}^{k+1,\beta,n} \sim x_{1,3,i}^{k+1,\alpha,n} \sim x_{1,2,3,i}^{k+1,\alpha,\gamma,n} \sim x_{1,2,i}^{\alpha,\gamma,n} \sim \bar{v}.$$
	We remark that if $\beta=k+3$ and $\alpha=k+3$, then $x_{1,3,i}^{k+1,\beta,n}=x_{1,3,i}^{k+1,\alpha,n}$, in which case the operation of going from $x_{1,3,i}^{k+1,\beta,n}$ to $x_{1,3,i}^{k+1,\alpha,n}$ is suppressed since the two vertices are the same. This procedure is also implicitly assumed and used in similar instances in the cases that follow.
	
	Also, let the paths $T_4$ and $T_5$ be
	\begin{align*}
	T_4 &:= \bar{u} \sim x_{1,3}^{\beta,n} \sim x_{1,3}^{\alpha,n} \sim x_{1,2,3}^{\alpha,\gamma,n} \sim \bar{v} \\
	T_5 &:= \bar{u} \sim x_{1,2,3}^{k+1,\beta,n} \sim x_{1,2,3}^{k+1,\alpha,n} \sim x_{1,2}^{\alpha,n} \sim \bar{v}.
	\end{align*}
	
	We note that the paths $Q_i$, $T_4$ and $T_5$ are $k-1$ internally disjoint paths, and all their internal vertices either contain the entries $z_{k+3}$ or $z_{k+4}$ or one of the pairs from the missing pairs of indices, or else they are lacking the pair of indices $\{z_3, z_4\}$. Thus, they are also internally disjoint from the paths $T_1$, $T_2$ and $P_j$ for $j\in\{5,\ldots,k\}$ (if any).
	
	Finally, to construct the last required paths which are internally disjoint from all the others obtained so far,
	\begin{enumerate}[$\circ$]
		\item   if $\beta \in \{k+2,k+3\}$, we put $j=3$ and $j=4$ in (\ref{EqnStar}) above to get the two paths $P_3$ and $P_4$;
		\item   if $\beta=1$, we put $j=4$ in (\ref{EqnStar}) to get $P_4$ and consider the path $T_6$ given by
		$$T_6 := u \sim x_{1,3}^{k+1,n} \sim x_{1,3}^{k+2,n} \sim x_{2,3}^{k+2,n} \sim v.$$
	\end{enumerate}
	
	Thus, the required $2k-1$ internally disjoint paths from the vertices in $C_1$ to the vertices in $C_2$ have been obtained.
	
	\noindent\emph{Case III.} The same entry from $\{z_3,\ldots,z_k\}$ is missing in both $\bar{u}$ and $\bar{v}$.
	
	Let the entry $z_3$ be missing from both vertices $\bar{u}$ and $\bar{v}$. Two different subcases can arise.
	\begin{enumerate}[(A)]
		\item   $\bar{u}$ is given by $x_{1,3}^{k+1,\beta}$ where $\beta\in\{1,k+2\}$, and $\bar{v}$ is given by $x_{2,3}^{k+2,k+3}$;
		\item   $\bar{u}$ is given by $x_{1,3}^{k+1,k+3}$, and $\bar{v}$ is given by $x_{2,3}^{k+2,\alpha}$ where $\alpha\in\{2,k+1,k+3,k+4\}$.
	\end{enumerate}
	
	\noindent\emph{Subcase III(A).}\\
	For $i\in\{4,\ldots,k\}$, let the paths $Q_i$ (if any) be given by
	$$Q_i:= \bar{u} \sim x_{1,3,i}^{k+1,\beta,n} \sim x_{1,2,3,i}^{k+1,k+3,\beta,n} \sim x_{2,3,i}^{k+2,k+3,n} \sim \bar{v}$$
	and the paths $T_7$ and $T_8$ be defined as
	\begin{align*}
	T_7 &:= \bar{u} \sim x_{1,3}^{\beta,n} \sim x_{1,3}^{k+3,n} \sim x_{1,2,3}^{k+2,k+3,n} \sim \bar{v}\\
	T_8 &:= \bar{u} \sim x_{1,2,3}^{k+1,\beta,n} \sim x_{1,2,3}^{k+1,k+3,n} \sim x_{2,3}^{k+3,n} \sim \bar{v}.
	\end{align*}
	As before, these $k-1$ new paths are mutually internally disjoint and are also internally disjoint from the paths $T_1$, $T_2$ and $P_j$ for $j\in\{5,\ldots,k\}$  (if any).
	
	To conclude this subcase, we construct the last required paths which are internally disjoint from all the others obtained so far,
	\begin{enumerate}[$\circ$]
		\item   if $\beta=k+2$, we put $j=3$ and $j=4$ in (\ref{EqnStar}) above to get the two paths $P_3$ and $P_4$;
		\item   if $\beta=1$, we put $j=4$ in (\ref{EqnStar}) to get $P_4$ and take the path $T_6$ as in Case II above.
	\end{enumerate}
	
	\noindent\emph{Subcase III(B).}\\
	For $i\in\{4,\ldots,k\}$, let the paths $Q_i$ (if any) be given by
	$$Q_i := \bar{u} \sim x_{1,3,i}^{k+1,k+3,n} \sim x_{3,i}^{k+3,n} \sim x_{2,3,i}^{k+2,k+3,n} \sim x_{2,3,i}^{k+2,\alpha,n} \sim \bar{v},$$
	which are mutually internally disjoint and are also internally disjoint from the paths $T_1$, $T_2$ and $P_j$ for $j\in\{5,\ldots,k\}$ (if any).
	
	Finally, the last required paths are constructed as follows.
	\begin{enumerate}[$\circ$]
		\item   If $\alpha\in\{k+3,k+4\}$, put $j=3$ and $ j=4 $ in (\ref{EqnStar}) above and consider the paths $T_9$ and $T_{10}$ defined by:
		\begin{align*}
		T_9 &:= \bar{u} \sim x_{1,3}^{k+3,n} \sim x_{1,3}^{\alpha,n} \sim x_{1,2,3}^{k+2,\alpha,n} \sim \bar{v}\\
		T_{10} &:= \bar{u} \sim x_{1,2,3}^{k+1,k+3,n} \sim x_{1,2,3}^{k+1,\alpha,n} \sim x_{2,3}^{\alpha,n} \sim \bar{v}
		\end{align*}
		\item   If $\alpha\in\{2,k+1\}$, put $j=4$ in (\ref{EqnStar}) to get $ P_4 $, and consider the paths $T_{11}$ $T_{12}$ and $ T_{13} $ defined by:
		\begin{align*}
		T_{11} &:= \bar{u} \sim x_{1,2,3}^{k+3,\alpha,n} \sim x_{1,2,3}^{k+2,\alpha,n} \sim \bar{v}\\
		T_{12} &:= \bar{u} \sim x_{1,2,3}^{k+3,\bar{\alpha},n} \sim x_{2,3}^{k+3,n} \sim x_{2,3}^{\alpha,n} \sim \bar{v}\\
		T_{13} &:= u \sim x_{1,3}^{k+1,n} \sim x_{2,3}^{\bar{\alpha},n} \sim x_{2,3}^{k+2,n} \sim v
		\end{align*}
		where $\bar{\alpha}=\{2,k+1\}\setminus\{\alpha\}$.
	\end{enumerate}
	
	\noindent\emph{Case IV.} A different entry from $\{z_3,\ldots,z_k\}$ is missing from $\bar{u}$ and $\bar{v}$.
	
	Let the entry $z_3$ be missing from the vertex $\bar{u}$ and the entry $z_4$ be missing from the vertex $\bar{v}$. This gives rise to three different subcases.
	\begin{enumerate}[(A)]
		\item   $\bar{u}$ is given by $x_{1,3}^{k+1,\beta}$ where $\beta\in\{1,k+2\}$, and $\bar{v}$ is given by $x_{2,4}^{k+2,k+3}$;
		\item   $\bar{u}$ is given by $x_{1,3}^{k+1,\beta}$ where $\beta\in\{1,k+2\}$, and $\bar{v}$ is given by $x_{2,4}^{k+2,\alpha}$ where $\alpha\in\{2,k+1\}$;
		\item   $\bar{u}$ is given by $x_{1,3}^{k+1,k+3}$, and $\bar{v}$ is given by $x_{2,4}^{k+2,\alpha}$ where $\alpha\in\{2,k+1,k+3,k+4\}$;
	\end{enumerate}
	
	In each of the following subcases, it can be checked that the new paths constructed are mutually internally disjoint and are also internally disjoint from the paths $T_1$, $T_2$ and $P_j$ for $j\in\{5,\ldots,k\}$ (if any).
	
	\noindent\emph{Subcase IV(A).}\\
	For $i\in\{5,\ldots,k\}$, let the paths $Q_i$ (if any) be given by
	$$Q_i:= \bar{u} \sim x_{1,3,i}^{k+1,\beta,n} \sim x_{1,2,3,i}^{k+1,k+3,\beta,n} \sim x_{1,2,i}^{k+3,\beta,n} \sim x_{2,4,i}^{k+2,k+3,n} \sim \bar{v}$$
	and the paths $T_{14}$, $T_{15}$ and $T_{16}$ be defined as
	\begin{align*}
	T_{14} &:= \bar{u} \sim x_{1,3,4}^{k+1,\beta,n} \sim x_{1,3,4}^{k+3,\beta,n} \sim x_{2,3,4}^{k+2,k+3,n} \sim \bar{v}\\
	T_{15} &:= \bar{u} \sim x_{1,3}^{\beta,n} \sim x_{3}^{n} \sim x_{4}^{n}\sim x_{2,4}^{k+3,n} \sim \bar{v}\\
	T_{16} &:= \bar{u} \sim x_{1,2,3}^{k+1,\beta,n} \sim x_{1,2,3}^{k+3,\beta,n} \sim x_{1,2,3}^{k+2,k+3,n} \sim x_{1,2,4}^{k+2,k+3,n} \sim \bar{v}
	\end{align*}
	
	To conclude this subcase,
	\begin{enumerate}[$\circ$]
		\item if $\beta=k+2$, put $j=3$ and $j=4$ in (\ref{EqnStar}) to get the paths $ P_3 $ and $ P_4 $;
		\item if $\beta=1$,  put $j=4$ in (\ref{EqnStar}) to get the path $ P_4 $, and take the path $T_6$ as in Case II above.
	\end{enumerate}

	\noindent\emph{Subcase IV(B).}\\
	For $i\in\{5,\ldots,k\}$, let the paths $Q_i$ (if any) be given by
	$$Q_i:= \bar{u} \sim x_{1,3,i}^{k+1,\beta,n} \sim x_{1,3,i}^{k+1,k+2,n} \sim x_{1,4,i}^{k+1,k+2,n} \sim x_{2,4,i}^{k+2,\alpha,n} \sim \bar{v}$$
	and the paths $T_{17}$, $T_{18}$, $T_{19}$ and $ T_{20} $ be defined as
	\begin{align*}
	T_{17} &:= \bar{u} \sim x_{1,3,4}^{k+1,\beta,n} \sim x_{1,3,4}^{k+1,k+2,n} \sim x_{2,3,4}^{k+2,\alpha,n} \sim \bar{v}\\
	T_{18} &:= \bar{u} \sim x_{1,2,3}^{k+1,\beta,n} \sim x_{1,2,3}^{k+1,k+2,n} \sim x_{1,2,4}^{k+1,k+2,n} \sim x_{1,2,4}^{k+2,\alpha,n} \sim \bar{v}\\
	T_{19} &:= \bar{u} \sim x_{1,3}^{\beta,n} \sim x_{3}^{n} \sim x_{4}^{n} \sim x_{2,4}^{\alpha,n} \sim \bar{v}
	\end{align*}
	
	To conclude this subcase,
	\begin{enumerate}[$\circ$]
		\item if $\alpha=2$ and $ \beta=1 $, put $j=4$  in (\ref{EqnStar}) to get $ P_4 $ and consider the path $ T_{20} $ defined by:
		$$ T_{20} := u \sim x_{1,3}^{k+1,n} \sim \omega \sim x_{2,3}^{k+2,n} \sim v $$
		where $ \omega = x_{1,3}^{k+2,n} $ when $ \beta=1 $ and $ \omega = x_{2,3}^{k+1,n} $ when $ \beta=k+2 $.
		
		\item if $\alpha=k+1$ and if $ \beta=k+2 $,  put $j=3$ in (\ref{EqnStar}) to get $ P_3 $ and consider the path $T_{21}$ defined by:
		$$T_{21} := u \sim x_{1,4}^{k+1,n} \sim x_{1,4}^{k+2,n} \sim x_{2,4}^{k+2,n} \sim v.$$
		\item if $\alpha=k+1$ and if $ \beta=1 $,  consider the paths $ T_6 $ and $T_{21}$.
	\end{enumerate}

	\noindent\emph{Subcase IV(C).}\\
	For $i\in\{5,\ldots,k\}$, let the paths $Q_i$ (if any) be given by
	$$Q_i:= \bar{u} \sim x_{1,3,i}^{k+1,k+3,n} \sim x_{1,3,i}^{k+2,k+3,n} \sim x_{1,4,i}^{k+2,k+3,n} \sim x_{2,4,i}^{k+2,k+3,n} \sim x_{2,4,i}^{k+2,\alpha,n} \sim \bar{v}$$
	and the path $T_{22}$ be defined as
	\begin{align*}
	T_{22} &:= \bar{u} \sim x_{1,3}^{k+3,n} \sim x_{2,3}^{k+3,n} \sim x_{2,3}^{\alpha,n} \sim x_{2,4}^{\alpha,n} \sim \bar{v}\\
	\end{align*}
	
	To conclude this last subcase,
	\begin{enumerate}[$\circ$]
		\item if $\alpha\in\{k+3,k+4\}$, put $j=3$ and $j=4$ in (\ref{EqnStar}) to get the paths $ P_3 $ and $ P_4$, and consider the paths $T_{23}$ and $T_{24}$ defined below:
		\begin{align*}
		T_{23} &:= \bar{u} \sim x_{2,3,4}^{k+1,k+3,n} \sim x_{2,3,4}^{k+1,\alpha,n} \sim x_{2,3,4}^{k+2,\alpha,n} \sim \bar{v}\\
		T_{24} &:= \bar{u} \sim x_{1,2,3}^{k+1,k+3,n} \sim x_{1,2,3}^{k+1,\alpha,n} \sim x_{1,2,4}^{k+1,\alpha,n} \sim x_{1,2,4}^{k+2,\alpha,n} \sim \bar{v}
		\end{align*}
		
		\item if $\alpha=2$,  put $j=3$ and $j=4$ in (\ref{EqnStar}) to get the paths $ P_3 $ and $ P_4$, and consider the paths $T_{25}$ and $T_{26}$ defined by:
		\begin{align*}
		T_{25} &:= \bar{u} \sim x_{2,3,4}^{k+1,k+3,n} \sim x_{2,3,4}^{k+3,\alpha,n} \sim x_{2,3,4}^{k+2,\alpha,n} \sim \bar{v}\\
		T_{26} &:= \bar{u} \sim x_{1,2,3}^{k+1,k+3,n} \sim x_{1,2}^{k+3,n} \sim x_{1,2,4}^{k+3,\alpha,n} \sim x_{1,2,4}^{k+2,\alpha,n} \sim \bar{v}.
		\end{align*}
		
		\item if $\alpha=k+1$, consider the paths $T_6$, $T_{21}$, $T_{25}$ and $T_{26}$.
	\end{enumerate}

	
\end{proof}

In Theorem \ref{Theo:k=3} we use induction on $n$ to show that the super--connectivity of $J(n,3)$ is equal to $5n-18$ for $n\geq 6$. This will serve in establishing the base case for the inductive argument used in proving our main result given in Theorem \ref{Theo:main}.

\begin{theorem} \label{Theo:k=3}
	$ \kappa'(J(n,3)) = 5n-18$, where $ n \geq 6 $.
\end{theorem}
\begin{proof}
	The upper bound follows from Lemma \ref{Lem:UB-J(n,k)}. For the lower bound, we use induction on $ n $, where $ n\geq 6 $. To establish the base case, it can be readily checked that $ \kappa'(J(6,3)) = 12 $. We assume that $ \kappa'(J(t,3)) \geq 5t-18 $ for $ 7\leq t \leq n-1 $.
	
	The proof of the inductive step is the same as that in Theorem \ref{Theo:main} by taking $k=3$, and hence we omit it here.	Letting $ S $ be a minimum super vertex of $ J(n,3) $, we have $ \kappa'(J(n,3))\geq |S|\geq 5n-18 $, as required.
	
	%
	%

\end{proof}

We conclude this section by proving our main result.

\begin{theorem} \label{Theo:main}
	$\kappa'(J(n,k))= (2k-1)(n-k)-k$ for $k\geq 3$ and $n\geq k+3$. \end{theorem}

\begin{proof}
	The upper bound follows from Lemma \ref{Lem:UB-J(n,k)}. For the lower bound, we use induction on $n$, where $n\geq k+3$ and $k\geq 3$. To establish the base case, we note that $J(n,k)$ is isomorphic to $J(n,n-k)$, and thus $\kappa'(J(k+3,k))=\kappa'(J(k+3,3))=5k-3=(2k-1)(n-k)-k$ for $n=k+3$ by Theorem \ref{Theo:k=3}. We assume that $\kappa'(J(t,k))\geq (2k-1)(t-k)-k$ for $k+3\leq t\leq n-1$.
	
	Let $ \mathcal{G}= J(n,k) $. Let $ S $ be a minimum super vertex--cut of $ \mathcal{G} $. Let $ C $ be a smallest component of $ \mathcal{G}-S $ and let $ C_*=(\mathcal{G}-S) -C$. By Lemma \ref{lemma5}, there is an entry $z_r$ for $r\in [n]$ that is not contained in every vertex of both $C$ and $C_*$. Let $x\in V(C)$ and let $y\in V(C_*)$ such that the vertices $x$ and $y$ do not contain the entry $z_r$.
	
	Suppose that all the neighbours of $x$ in $\mathcal{G}-S$ contain the  entry $z_r$. Then by Lemma \ref{Lem:AllNborsContainR}, $|S|\geq (2k-1)(n-k)-k$,  and we are done. Similarly, there is nothing more to prove if all the neighbours of $y$ in $\mathcal{G}-S$ contain the entry $z_r$.
	
	Thus, we can assume that at least one of the neighbours of $x$ in $C$ and at least one of the neighbours of $y$ in $C_*$ do not contain the  entry $z_r$. We consider $\mathcal{G}'=\mathcal{G}-\beta_r$, which is isomorphic to $J(n-1,k)$. If we let $C'=C-\beta_r$ and $C'_{*}=C_{*}-\beta_r$, then $|C'|\geq 2$ and $|C'_*|\geq 2$. Also, by Lemma \ref{lemma4}, the entry $z_r$ cannot be contained in every vertex of $S$. Thus, $S'=S-\beta_r$ is not empty and we have $S'\subset S$ by Lemma \ref{lemma3}. Thus, $S'$ is a minimum super vertex-cut of $\mathcal{G}$ since it is non-empty and there are no isolated vertices in $\mathcal{G} - S'$. Hence, by the inductive hypothesis, $|S'|\geq (2k-1)(n-1-k)-k$. By Lemma \ref{Lem:S-S'} we have that $ |S \cap \beta_r| \geq 2k-1 $, implying that $|S-S'|\geq 2k-1$. This completes our proof.
	
	Hence, $\kappa'(J(n,k))\geq |S|\geq (2k-1)(n-k)-k$, as required.
	
\end{proof}

\section{Conclusion} \label{Sect-Conclusion}

As already discussed, the Johnson graph $J(n,k)$ is isomorphic to $J(n,n-k)$, and thus $J(k+1,k)\cong J(k+1,1)$ and $J(k+2,k)\cong J(k+2,2)$. Hence, the super--connectivity of the family of Johnson graphs is given by the proposition below.

\begin{proposition}\label{FinalProp}
	The super--connectivity of the Johnson graph $J(n,k)$ for $n\geq k\geq 1$ is given by
	$$\kappa'(J(n,k))=\left\{
	\begin{array}{cl}
	3(n-3) & $if $ k=2 $ and $ n\geq 6,\\
	3(k-1) & $if $ k\geq 4 $ and $ n=k+2,\\
	(2k-1)(n-k)-k & $if $ k\geq 3 $ and $ n\geq k+3\\
	+\infty & $otherwise.$
	\end{array}\right.$$
\end{proposition}

In future work, it would be interesting to investigate the super--connectivity of the uniform subset graphs $G(n,k,t)$ for other values of the parameter $t$. It is worth noting that if $t=1$, then $G(n,k,1)$ is the Kneser graph $KG(n,k)$. \cite{GauciBoruzanli} proved that the super--connectivity of $KG(n,2)$ for $n\geq 5$ is equal to ${n \choose 2}-6$, and a related conjecture for the remaining cases was formulated.

\acknowledgements
\label{sec:ack}
The authors would like to thank the referees for their comments and suggestions which helped in improving this paper.

\nocite{*}
\bibliographystyle{abbrvnat}
\bibliography{dmtcs}
\label{sec:biblio}

\end{document}